\documentclass[12pt]{amsart}
\usepackage{amscd,amsmath,amssymb,amsfonts}
\usepackage[all]{xy}
\theoremstyle{plain}
\newtheorem{thm}{Theorem}
\newtheorem{lem}[thm]{Lemma}
\newtheorem{cor}[thm]{Corollary}

\theoremstyle{definition}

\newtheorem{rmk}[thm]{Remark}

\numberwithin{thm}{section} \numberwithin{equation}{section}

\newcommand{\ga}[2]{\begin{gather}\label{#1}#2 \end{gather}}

\newcommand{\sE}{{\mathcal E}}
\newcommand{\sF}{{\mathcal F}}

\newcommand{\sL}{{\mathcal L}}

\newcommand{\sO}{{\mathcal O}}

\begin{document}

\title{Stability of sheaves of locally closed and exact forms }
\author{Xiaotao Sun}
\address{Academy of Mathematics and Systems Science, Chinese Academy of Science, Beijing, P. R. of China}
\email{xsun@math.ac.cn}
\address{}
\date{April 23, 2009}
\thanks{Partially supported by NSFC (No. 10731030) and Key Laboratory of Mathematics
Mechanization (KLMM)}

\begin{abstract} For any smooth projective variety $X$ of dimension $n$ over an
algebraically closed field $k$ of characteristic $p>0$ with
$\mu(\Omega^1_X)>0$. If ${\rm T}^{\ell}(\Omega^1_X)$
($0<\ell<n(p-1)$) are semi-stable, then the sheaf $B^1_X$ of exact
$1$-forms is stable. When $X$ is a surface with $\mu(\Omega^1_X)>0$
and $\Omega^1_X$ is semi-stable, the sheaf $B^2_X$ of exact
$2$-forms is also stable. Moreover, under the same condition, the
sheaf $Z^1_X$ of closed $1$-forms is stable when $p>3$, and $Z^1_X$
is semi-stable when $p=3$.
\end{abstract}
\maketitle

\section{Introduction}

Let $X$ be a smooth projective variety over an algebraically closed
field $k$ with ${\rm char}(k)=p>0$. Let $F:X\to X_1:=X\times_kk$
denote the relative Frobenius morphism over $k$. In the de Rham
complex of $X$, the subsheaf $B^i_X={\rm image}\,
(d:F_*\Omega^{i-1}_X\to F_*\Omega^i_X)$ (resp. $Z^i_X={\rm kernel}\,
(d:F_*\Omega^i_X\to F_*\Omega^{i-1}_X)$) of $F_*\Omega^i_X$ is
called the sheaf of locally exact $i$-forms (resp. locally closed
$i$-forms). Fix an ample divisor ${\rm H}$ on $X$, the slope of a
torsion free sheaf $\sE$ is $\mu(\sE):=c_1(\sE)\cdot {\rm
H}^{n-1}/{\rm rk}(\sE)$, where ${\rm dim}(X)=n$ and ${\rm rk}(\sE)$
denotes the rank of $\sE$. A torsion free sheaf $\sE$ is called
semi-stable (resp. stable) if $\mu(\sE')\le\mu(\sE)$ (resp.
$\mu(\sE')<\mu(\sE)$) for any nontrivial proper sub-sheaf
$\sE'\subset\sE$. In this notes, we prove some observations about
stability of $B^i_X$ and $Z^i_X$.

When $X$ is a smooth projective curve of genus $g\ge 2$,  the
semi-stability of $B^1_X$ is proved in \cite{Ra}, and its stability
is proved in \cite{J}. When $X$ is a smooth projective surface with
semi-stable $\Omega^1_X$ and $\mu(\Omega^1_X)>0$, the semi-stability
of $B^1_X$ and $B^2_X$ is proved in \cite{Kit}. But it is not known
whether $Z^1_X$ is semi-stable (cf. \cite[Remark 3.4]{Kit}).

We show firstly that the sheaf $B^1_X$ of local exact differential
$1$-forms on $X$ is stable if $\mu(\Omega^1_X)>0$ and ${\rm
T}^{\ell}(\Omega^1_X)$ ($0<\ell<n(p-1)$) are semi-stable. For
surfaces, the semi-stability of $\Omega^1_X$ implies semi-stability
of ${\rm T}^{\ell}(\Omega^1_X)$ ($0<\ell<2(p-1)$). Thus our result
is a generalization of the results in \cite{J} and \cite{Kit}. Then
we show secondly that $B_X^2$ and $Z^1_X$ are stable when $X$ is a
smooth projective surface with semi-stable $\Omega^1_X$ and
$\mu(\Omega^1_X)>0$. It solves in particular the open problem in
\cite{Kit}.

\section{The stability of $B_X^1$}

Let $X$ be a smooth projective variety of dimension $n$. To study
the stability of $B_X^1$, we recall from \cite{Su} that there is a
filtration
 \ga{2.1}{0=V_{n(p-1)+1}\subset
V_{n(p-1)}\subset\cdots\subset V_1\subset V_0=V=F^*(F_*W)} with
injective homomorphisms $\nabla:V_{\ell}/V_{\ell+1}\to
V_{\ell-1}/V_{\ell}\otimes\Omega^1_X$ such that
$$\nabla^{\ell}:V_{\ell}/V_{\ell+1}\cong W\otimes{\rm
T}^{\ell}(\Omega^1_X)\,\,,\,\,\quad 0\le\ell\le n(p-1).$$

Let $\sE\subset F_*W$ be a nontrivial subsheaf, the canonical
filtration \eqref{2.1} induces the filtration (we assume $V_m\cap
F^*\sE\neq 0$) \ga{2.2}{0\subset V_m\cap
F^*\sE\subset\,\cdots\,\subset V_1\cap F^*\sE\subset V_0\cap
F^*\sE=F^*\sE.} Let
$$\sF_{\ell}:=\frac{V_{\ell}\cap F^*\sE}{V_{\ell+1}\cap
F^*\sE}\subset\frac{V_{\ell}}{V_{\ell+1}}, \qquad r_{\ell}={\rm
rk}(\sF_{\ell}).$$ Then $$\mu(F^*\sE)=\frac{1}{{\rm
rk}(\sE)}\sum_{\ell=0}^mr_{\ell}\cdot\mu(\sF_{\ell})$$ and
\ga{2.3}{\mu(\sE)-\mu(F_*W)=\frac{1}{p\cdot{\rm
rk}(\sE)}\sum^m_{\ell=0}r_{\ell}\left(\mu(\sF_{\ell})-\mu(F^*F_*W)\right).}
By using the following formula (cf. Lemma 4.2 and Lemma 4.3 in
\cite{Su}) \ga{2.4} {\mu(F^*F_*W)=p\cdot\mu(F_*W)=
\frac{p-1}{2}K_X\cdot{\rm
H}^{n-1}+\mu(W), \\
 \mu(V_{\ell}/V_{\ell+1})=\mu(W\otimes {\rm
T^{\ell}}(\Omega^1_X))=\frac{\ell}{n}K_X\cdot{\rm H}^{n-1}+\mu(W).
\notag} we can formulate the following lemma

\begin{lem}\label{lem2.1} The morphisms $\nabla:V_{\ell}/V_{\ell+1}\to
V_{\ell-1}/V_{\ell}\otimes\Omega^1_X$ induce injective morphisms
$\nabla: \sF_{\ell}\to \sF_{\ell-1}\otimes\Omega^1_X$. Moreover, we
have  \ga{2.5}{\mu(\sE)-\mu(F_*W)=
\sum^m_{\ell=0}r_{\ell}\frac{\mu(\sF_{\ell})
-\mu(\frac{V_{\ell}}{V_{\ell+1}})}{p\cdot{\rm rk}(\sE)}\\-
\notag\frac{\mu(\Omega^1_X)}{p\cdot{\rm
rk}(\sE)}\sum^m_{\ell=0}(\frac{n(p-1)}{2}-\ell)r_{\ell}}
\end{lem}

When $m\le\frac{n(p-1)}{2}$, it is clear that \ga{2.6}
{\sum^m_{\ell=0}(\frac{n(p-1)}{2}-\ell)r_{\ell}\ge\frac{n(p-1)}{2}r_0\ge\frac{n(p-1)}{2}.}
When $m>\frac{n(p-1)}{2}$, we can write \ga{2.7}
{\sum_{\ell=0}^m(\frac{n(p-1)}{2}-\ell)r_{\ell}=\sum^{n(p-1)}_{\ell=m+1}
(\ell-\frac{n(p-1)}{2})
r_{n(p-1)-\ell}\\+\sum^m_{\ell\,>\frac{n(p-1)}{2}}(\ell-\frac{n(p-1)}{2})
(r_{n(p-1)-\ell}-r_{\ell}).\notag} The fact that
$V_{\ell}/V_{\ell+1}\xrightarrow{\nabla}(V_{\ell-1}/V_{\ell})\otimes\Omega^1_X$
induce injective morphisms
$\sF_{\ell}\xrightarrow{\nabla}\sF_{\ell-1}\otimes\Omega^1_X\quad
(1\le\ell\le m)$ implies the following inequalities
$$r_{n(p-1)-\ell}-r_{\ell}\ge 0 \qquad (\ell>\frac{n(p-1)}{2})$$
(cf. \cite[Proposition 4.7]{Su}). Thus we have \ga{2.8}
{\sum^m_{\ell=0}(\frac{n(p-1)}{2}-\ell)r_{\ell}\ge\frac{n(p-1)}{2}r_0\quad
{\rm when}\,\,\, m\neq n(p-1)\,.}

\begin{lem}\label{lem2.2} When $m=n(p-1)$, we have
 $$r_{\ell}\ge r_{n(p-1)}\cdot {\rm rk}({\rm
T}^{n(p-1)-\ell}(\Omega^1_X))$$ which implies the following
inequality
$${\rm rk}(\sE)=\sum^m_{\ell=0}r_{\ell}\ge
r_{n(p-1)}\sum^m_{\ell=0}{\rm rk}({\rm
T}^{n(p-1)-\ell}(\Omega^1_X))=r_{n(p-1)}\cdot p^n$$
\end{lem}
\begin{proof}
It is a local problem to prove the lemma. Let $K=K(X)$ be the
function field of $X$ and consider the $K$-algebra
$$R=\frac{K[\alpha_1,\cdots,\alpha_n]}{(\alpha_1^p,\cdots,\alpha_n^p)}=\bigoplus^{n(p-1)}_{\ell=0}R^{\ell},$$
where $R^{\ell}$ is the $K$-linear space generated by
$$\{\,\alpha_1^{k_1}\cdots \alpha_n^{k_n}\,|\,k_1+\cdots+k_n=\ell,\quad 0\le k_i\le
p-1\,\}.$$ The quotients in the filtration \eqref{2.1} can be
described locally
$$V_{\ell}/V_{\ell+1}=W\otimes_K R^{\ell}$$
as $K$-vector spaces. If $K=k(x_1,..., x_n)$, then the homomorphism
$$\nabla: W\otimes_K R^{\ell}\to W\otimes_KR^{\ell-1}\otimes_K\Omega^1_{K/k}$$
is locally the $k$-linear homomorphism (cf. (3.6) in \cite{Su})
defined by
$$\nabla(w\otimes\alpha_1^{k_1}\cdots\alpha_n^{k_n})=-w\otimes\sum^n_{i=1}k_i
(\alpha_1^{k_1}\cdots\alpha_i^{k_i-1}\cdots\alpha_n^{k_n})
\otimes_K{\rm d}x_i.$$ Then the fact that
$\sF_{\ell}\xrightarrow{\nabla}\sF_{\ell-1}\otimes \Omega^1_X$ for
$\sF_{\ell}\subset W\otimes R^{\ell}$ is equivalent to
\ga{2.9}{\forall\,\,\, \sum_jw_j\otimes f_j \in
\sF_{\ell}\,\,\Rightarrow\,\,\sum_jw_j\otimes\frac{\partial
f_j}{\partial\alpha_i}\, \in\,\sF_{\ell-1}\quad (1\le i\le n).}

The polynomial ring ${\rm P}=
K[\partial_{\alpha_1},\cdots,\partial_{\alpha_n}]$ acts on $R$
through partial derivations, which induces a ${\rm D}$-module
structure on $R$, where
$${\rm D}=\frac{K[\partial_{\alpha_1},\cdots,\partial_{\alpha_n}]}{(\partial_{\alpha_1}^p,\cdots,\partial_{\alpha_n}^p)}
=\bigoplus^{n(p-1)}_{\ell=0}{\rm D}_{\ell}$$ and ${\rm D}_{\ell}$ is
the linear space of degree $\ell$ homogeneous elements. In
particular, $W\otimes R$ has the induced ${\rm D}$-module structure
with ${\rm D}$ acts on $W$ trivially. Use this notation, \eqref{2.9}
is equivalent to ${\rm D}_1\cdot \sF_{\ell}\subset \sF_{\ell-1}$.

Since $R^{n(p-1)}$ is of dimension $1$, for any subspace
$$\sF_{n(p-1)}\subset W\otimes R^{n(p-1)},$$ there is a subspace
$W'\subset W$ of dimension $r_{n(p-1)}$ such that
$$\sF_{n(p-1)}=W'\otimes R^{n(p-1)}.$$
Thus ${\rm D}_{\ell}\cdot\sF_{n(p-1)}=W'\otimes {\rm D}_{\ell}\cdot
R^{n(p-1)}=W'\otimes R^{n(p-1)-\ell}\subset\sF_{n(p-1)-\ell}\,$ for
all $0\le\ell\le n(p-1)$, which proves the lemma.
\end{proof}

Recall that the sheaf $B^1_X$ of locally exact differential forms on
$X$ is defined by exact sequence \ga{2.10} {0\to \sO_X\to F_*\sO_X
\to B^1_X\to 0.}

\begin{thm}\label{thm2.3} Let $X$ be a smooth projective variety of
dimension $n$ with $\mu(\Omega^1_X)>0$ and $\sL$ a torsion free
sheaf of rank $1$ on $X$. Assume that ${\rm T}^{\ell}(\Omega^1_X)$
($0<\ell<n(p-1)$) are semi-stable. Then, for any nontrivial
subsheaves $\sE\subset F_*\sL$ and $B'\subset B^1_X$, we have
\ga{2.11} {\mu(\sE)-\mu(F_*\sL)\le -
\frac{\mu(\Omega^1_X)}{p\cdot{\rm rk}(\sE)}\cdot\frac{n(p-1)}{2}}
\ga{2.12}{\mu(B')-\mu(B^1_X)\le
-\frac{\mu(\Omega^1_X)}{p\cdot(p^n-1)}\cdot\frac{n(p-1)}{2}}  when
${\rm rk}(\sE)<{\rm rk}(F_*\sL)$ and ${\rm rk}(B')<{\rm rk}(B_X^1)$.
\end{thm}

\begin{proof} Since ${\rm rk}(\sE)<{\rm rk}(F_*\sL)=p^n$, by Lemma
\ref{lem2.2}, $m\neq n(p-1)$. On the other hand, when $\sL$ is of
rank $1$, $V_{\ell}/V_{\ell+1}\cong \sL\otimes {\rm
T}^{\ell}(\Omega^1_X)$ are semi-stable by the assumption. Thus, by
\eqref{2.5}, we have \ga{2.13}{\mu(\sE)-\mu(F_*\sL)\le -
\frac{\mu(\Omega^1_X)}{p\cdot{\rm
rk}(\sE)}\sum^m_{\ell=0}(\frac{n(p-1)}{2}-\ell)r_{\ell}} which
implies \eqref{2.11} by \eqref{2.8} since $m\neq n(p-1)$.

To show \eqref{2.12}, for $B'\subset B^1_X$ of rank $r<{\rm
rk}(B^1_X)$, let $\sE\subset F_*\sO_X$ be the subsheaf of rank $r+1$
such that we have exact sequence
$$0\to \sO_X\to\sE\to B'\to 0.$$
Substitute \eqref{2.11} to
$\mu(B')-\mu(B^1_X)=\frac{r+1}{r}\mu(\sE)-\frac{p^n}{p^n-1}\mu(F_*\sO_X)$,
we have $$\mu(B')-\mu(B^1_X)\le
\frac{p^n-1-r}{r(p^n-1)}\mu(F_*\sO_X)-\frac{n(p-1)}{2rp}\mu(\Omega^1_X).$$
By the formula \eqref{2.4}, we have
$\mu(F_*\sO_X)=\frac{n(p-1)}{2p}\mu(\Omega^1_X)$. Thus
$$\mu(B')-\mu(B^1_X)\le
-\frac{\mu(\Omega^1_X)}{p\cdot(p^n-1)}\cdot\frac{n(p-1)}{2}.$$
\end{proof}

\begin{cor}\label{cor2.4} Let $X$ be a smooth projective curve of
genus $g\ge 2$. Then, for all proper sub-bundles $\sE\subset
F_*\sL$, $B'\subset B^1_X$, we have
$$\mu(\sE)-\mu(F_*\sL)\le-\,\frac{p-{\rm rk}(\sE)}{p}(g-1)$$
$$\mu(B')-\mu(B^1_X)\le -\,\frac{p-1-{\rm rk}(B')}{p}(g-1).$$
\end{cor}
\begin{proof} When ${\rm dim}(X)=1$,
$V_{\ell}/V_{\ell+1}=\sL\otimes\omega_X^{\ell}$ are line bundles and
thus $r_{\ell}=1$ ($0\le \ell\le m$) in \eqref{2.5}. Then we can
rewrite \eqref{2.5}:
$$\mu(\sE)-\mu(F_*\sL)=\sum^m_{\ell=0}\frac{\mu(\sF_{\ell})-\mu(\frac{V_{\ell}}{V_{\ell+1}})}{p\cdot{\rm rk}(\sE)}
-\frac{(p-{\rm rk}(\sE))(g-1)}{p},$$ which impiles the following
inequality
$$\mu(\sE)-\mu(F_*\sL)\le-\frac{p-{\rm rk}(\sE)}{p}(g-1)$$
and the equality holds if and only if
$\sF_{\ell}=V_{\ell}/V_{\ell+1}$. Similarly, we have
$$\mu(B')-\mu(B^1_X)\le -\frac{p-1-{\rm rk}(B')}{p}(g-1).$$
\end{proof}

\begin{cor}\label{cor2.5} Let $X$ be a smooth projective surface
with $\mu(\Omega^1_X)>0$. If $\Omega^1_X$ is semi-stable,  then for
any proper nontrivial $B'\subset B^1_X$
$$\mu(B')-\mu(B^1_X)\le
-\frac{\mu(\Omega^1_X)}{p\cdot(p+1)}$$
\end{cor}

\begin{proof} When ${\rm dim}(X)=2$, we have (cf. Proposition 3.5 of \cite{Su})
$${\rm T}^{\ell}(\Omega^1_X)=\left\{
\begin{array}{llll}{\rm Sym}^{\ell}(\Omega^1_X)  &\mbox{when $\ell<p$}\\
{\rm Sym}^{2(p-1)-\ell}(\Omega^1_X)\otimes\omega_X^{\ell-(p-1)}
&\mbox{when $\ell\ge p$}
\end{array}\right.$$
where $\omega_X=\Omega^2_X=\sO_X(K_X)$ is the canonical line bundle
of $X$. Thus ${\rm T}^{\ell}(\Omega^1_X)$ are semi-stable whenever
$\Omega^1_X$ is semi-stable. Then the corollary follows the theorem.
\end{proof}

\begin{rmk}\label{rmk2.6} When $X$ is a curve of genus $g\ge 2$,
the stability of $B^1_X$ was proved by K. Joshi in \cite{J}. When
$X$ is a surface with $\mu(\Omega^1_X)>0$, if $\Omega^1_X$ is
semi-stable, Y. Kitadai and H. Sumihiro proved in \cite{Kit} that
$B^1_X$ is semi-stable.
\end{rmk}

\section{The stability of $B^2_X$ and $Z^1_X$}

Let $X$ be a smooth projective surface with $\mu(\Omega^1_X)>0$.
When $\Omega^1_X$ is semi-stable, Y. Kitadai and H. Sumihiro proved
in \cite{Kit} that $B^1_X$ and $B^2_X$ are semi-stable, but it is
left open whether $Z^1_X$ is semi-stable or not (cf. \cite[Remark
3.4]{Kit}). In this section, we consider the stability of  $B^2_X$
and $Z^1_X$.

Recall the definition of $B_X^i$, $Z^1_X$, consider de Rham complex
of $X$:
$$F_*\sO_X\xrightarrow{d_1}F_*\Omega^1_X\xrightarrow{d_2}F_*\Omega^2_X=F_*\omega_X$$
the vector bundles $B_X^i$ ($i=1,\,\,2$) and $Z_X^1$ are defined by
$$B^i_X:={\rm
image}\,(\,F_*\Omega^{i-1}_X\xrightarrow{d_i}F_*\Omega^i_X\,)\,\,,\,\,\,\,
Z_X^1:={\rm
kernel}\,(\,F_*\Omega^1_X\xrightarrow{d_2}F_*\Omega^2_X\,)\,.$$ By
the definition and Cartier isomorphism, these bundles are suited in
the following exact sequences \ga{3.1} {0\to\sO_X\to F_*\sO_X\to
B^1_X\to 0} \ga{3.2} {0\to Z^1_X\to F_*\Omega^1_X\to B^2_X\to 0}
\ga{3.3} {0\to B^1_X\to Z^1_X\to \Omega^1_X\to 0} \ga{3.4} {0\to
B^2_X\to F_*\omega_X\to \omega_X\to 0} where
$\omega_X=\Omega^2_X=\sO_X(K_X)$.

\begin{lem}\label{lem3.1} Let $X$ be a smooth projective surface with $\mu(\Omega^1_X)>0$.
If $\Omega^1_X$ is semistable, then, for any subsheaf $B'\subset
B^2_X$, we have \ga{3.5} {\mu(B')-\mu(B_X^2)\le -\,\frac{{\rm
rk}(B_X^2)-{\rm rk}(B')}{p(p+1){\rm rk}(B')}\mu(\Omega^1_X)\,.} In
particular, $B_X^2$ is stable.
\end{lem}

\begin{proof} For $B'\subset B_X^2$, by exact sequence \eqref{3.4}
and Theorem \ref{thm2.3}, \ga{3.6} {\mu(B')\le
\mu(F_*\omega_X)-\frac{p-1}{p\,{\rm rk}(B')}\mu(\Omega^1_X)\,.} By
the formula \eqref{2.4} and the exact sequence \eqref{3.4}, we have
\ga{3.7}{\mu(F_*\omega_X)=\frac{p+1}{p}\mu(\Omega^1_X)\,,\quad
\mu(B^2_X)=\frac{p+2}{p+1}\mu(\Omega^1_X)\,.} Substitute \eqref{3.7}
to \eqref{3.6}, we have the inequality \eqref{3.5}.

\end{proof}

\begin{thm}\label{thm3.2} Let $X$ be a smooth projective surface with
$\mu(\Omega^1_X)>0$. Then, when $\Omega^1_X$ is semistable, the
sheaf $Z^1_X$ of locally closed $1$-forms is stable when $p>3$.
$Z^1_X$ is semistable when $p=3$.
\end{thm}
\begin{proof} For $B\subset Z^1_X$ with ${\rm rk}(B)<{\rm rk}(Z_X^1)$, by
\eqref{3.3}, there are subsheaves $B'\subset B^1_X$,
$B''\subset\Omega^1_X$ satisfy the exact sequence \ga{3.8} {0\to
B'\to B\to B''\to 0.} If $B''=0$,
$\mu(B)=\mu(B')\le\mu(B_X^1)<\mu(Z_X^1)$. Thus we can assume
$B''\neq 0$. If ${\rm rk}(B')={\rm rk}(B^1_X)$, then ${\rm
rk}(B'')=1$ and $$\mu(B)-\mu(Z_X^1)\le
-\frac{p-1}{p^2(p^2+1)}\mu(\Omega^1_X).$$ Thus we can assume ${\rm
rk}(B')<{\rm rk}(B^1_X)$ if $B'\neq 0$. We complete the proof in the
following two lemmas, which deal with the cases $B'\neq0$ and $B'=0$
respectively.
\end{proof}

\begin{lem}\label{lem3.3} For $B\subset Z^1_X$ with ${\rm rk}(B)<{\rm rk}(Z_X^1)$, assume that $B'$ defined in
\eqref{3.8} is nontrivial with ${\rm rk}(B')<{\rm rk}(B^1_X)$. Then,
when $p>3$,
$$\mu(B)-\mu(Z^1_X)\le -\frac{\mu(\Omega^1_X)}{p(p^2+1){\rm
rk}(B)}$$ and, when $p=3$, $\mu(B)-\mu(Z^1_X)\le 0$.
\end{lem}

\begin{proof} Since $B'\neq 0$ with ${\rm rk}(B')<{\rm rk}(B^1_X)$, by \eqref{3.1}, there is a
subsheaf $\sE\subset F_*\sO_X$ with ${\rm rk}(\sE)<p^2$ satisfying
the exact sequence \ga{3.9} {0\to\sO_X\to\sE\to B'\to 0.} The
canonical filtration $0\subset V_{2p-2}\subset\cdots\subset
V_1\subset V_0=F^*F_*\sO_X$ induces
$$0\subset V_m\cap F^*\sE\subset \cdots\subset V_1\cap F^*\sE\subset
V_0\cap F^*\sE=F^*\sE$$ where $m$ is the maximal number such that
$V_m\cap F^*\sE\neq 0$. Let
$$\sF_{\ell}:=\frac{V_{\ell}\cap F^*\sE}{V_{\ell+1}\cap
F^*\sE}\subset\frac{V_{\ell}}{V_{\ell+1}}, \qquad r_{\ell}={\rm
rk}(\sF_{\ell}).$$ When ${\rm dim}(X)=2$, we have (cf. Proposition
3.5 of \cite{Su})
$${\rm T}^{\ell}(\Omega^1_X)=\left\{
\begin{array}{llll}{\rm Sym}^{\ell}(\Omega^1_X)  &\mbox{when $\ell<p$}\\
{\rm Sym}^{2(p-1)-\ell}(\Omega^1_X)\otimes\omega_X^{\ell-(p-1)}
&\mbox{when $\ell\ge p$}
\end{array}\right.$$
where $\omega_X=\sO_X(K_X)$ is the canonical line bundle of $X$.
Thus $V_{\ell}/V_{\ell+1}$ are semi-stable whenever $\Omega^1_X$ is
semi-stable. Then, by Lemma \ref{lem2.1},  \ga{3.10}
{\mu(\sE)-\mu(F_*\sO_X)\le - \frac{\mu(\Omega^1_X)}{p\cdot{\rm
rk}(\sE)}\sum^m_{\ell=0}(p-1-\ell)r_{\ell}} where, by Lemma
\ref{2.2}, we have $m<2p-2$ since ${\rm rk}(\sE)<p^2$. Notice that
$r_0=1$, then \eqref{3.10} implies \ga{3.11} {\mu(B')-\mu(B_X^1)\le
-\frac{\mu(\Omega^1_X)}{p\,{\rm
rk}(B')}\left(\sum^m_{\ell=1}(p-1-\ell)r_{\ell}+\frac{{\rm
rk}(B')}{p+1}\right).} By \eqref{3.1}, \eqref{3.3} and \eqref{3.8},
we have
$\mu(B^1_X)=\mu(Z^1_X)-\frac{2}{(p+1)(p^2+1)}\mu(\Omega^1_X)$,
\ga{3.12} {\mu(B)=\frac{{\rm rk}(B')}{{\rm rk}(B)}\mu(B')+\frac{{\rm
rk}(B'')}{{\rm rk}(B)}\mu(B'').} Substitute \eqref{3.11} and
$\mu(B'')\le\mu(\Omega^1_X)$ to \eqref{3.12}, we have
$$\aligned \mu(B)&\le \frac{{\rm rk}(B')}{{\rm
rk}(B)}\left(\mu(B_X^1)-\frac{\mu(\Omega^1_X)}{p\,{\rm
rk}(B')}(\Sigma_m+\frac{{\rm rk}(B')}{p+1})\right)+\frac{{\rm
rk}(B'')}{{\rm
rk}(B)}\mu(\Omega^1_X)\\
&=\frac{{\rm rk}(B')}{{\rm
rk}(B)}\mu(B_X^1)-\frac{\mu(\Omega^1_X)}{p\,\,{\rm
rk}(B)}\Sigma_m+\left({\rm rk}(B'')-\frac{{\rm
rk}(B')}{p(p+1)}\right)\frac{\mu(\Omega^1_X)}{{\rm rk}(B)}\\
&=\mu(Z_X^1)-\frac{\mu(\Omega^1_X)}{p(p^2+1){\rm
rk}(B)}\left\{\aligned&(p^2+1)\Sigma_m+(p+1){\rm
rk}(B')\\&-(p^2-p){\rm rk}(B'')\endaligned\right\}
\endaligned $$
where $\Sigma_m:=\sum^m_{\ell=1}(p-1-\ell)r_{\ell}\,\ge 0$ since
$m<2p-2$\,. Let
$$N:=(p^2+1)\Sigma_m+(p+1){\rm rk}(B')-(p^2-p){\rm
rk}(B'')$$ To prove the lemma, it is enough to prove the claim that
$N>0$ when $p>3$, and $N\ge 0$ when $p=3$. If $\Sigma_m\ge 2$, the
claim is clear since ${\rm rk}(B'')\le{\rm rk}(\Omega^1_X)=2$. Thus
we can assume that $\Sigma_m\le 1$. To prove the claim, we also
remark that $m\ge 1$ since $r_0=1$ and \ga{3.13}
{\sum^m_{\ell=0}r_{\ell}={\rm rk}(\sE)={\rm rk}(B')+1\ge 2\,.}

If $m\le p-1$, then $\Sigma_m\ge (p-2)r_1$. The condition
$\Sigma_m\le 1$ implies that $p=3$ and $\Sigma_m=1$. Thus
$N=10+4\,{\rm rk}(B')-6\,{\rm rk}(B'')>0$.

If $m>p-1$, we show firstly that the condition $\Sigma_m\le 1$
implies $m+1>2p-3$. In fact, it is clear when $p=3$. To show it for
the case $p>3$, by the formula \eqref{2.7}, we can write
$$\Sigma_m=\sum^{2p-3}_{\ell=m+1} (\ell-p+1)
r_{2p-2-\ell}+\sum^m_{\ell\,>p-1}(\ell-p+1)
(r_{2p-2-\ell}-r_{\ell})$$ where $r_{2p-2-\ell}\ge r_{\ell}$ (
$\ell>p-1$). Thus, if $m+1\le 2p-3$, we have the following
contradiction
$$\Sigma_m\ge (m-p)r_{2p-2-(m+1)}>r_{2p-2-(m+1)}\ge 1$$
where we remark that all of $r_{\ell}={\rm rk}(\sF_{\ell})$
($0\le\ell\le m$) are non-zero since the existence of injections
$\sF_{\ell}\to\sF_{\ell-1}\otimes\Omega^1_X$ (cf. Lemma
\ref{lem2.1}) and $r_m\neq 0$ (by definition). Then the fact that
$m+1>2p-3$ implies
$${\rm rk}(B')=\sum^m_{\ell=1}r_{\ell}\ge m\ge 2p-3.$$
Thus we have $$\aligned N&\ge (p+1)(2p-3)-(p^2-p){\rm
rk}(B'')\\&=(2-{\rm rk}(B''))p^2+({\rm rk}(B'')-1)p-3\endaligned$$
which is positive when $p>3$, and non-negative when $p=3$.
\end{proof}

\begin{lem}\label{lem3.4} If $B'=0$ and $p>2$, then we have
$$\mu(B)-\mu(Z^1_X)<0.$$
\end{lem}

\begin{proof} When $B'=0$, $B=B''$ has rank at most $2$. By
\eqref{3.2}, we consider $B\subset F_*\Omega^1_X$ and the canonical
filtration $$0=V_{2p+1}\subset V_{2p-2}\subset
\,\,\cdots\,\,V_1\subset V_0=F^*F_*\Omega^1_X\,.$$ If $V_1\cap
F^*B=0$, then $F^*B\subset V_0/V_1=\Omega^1_X$. By
$\mu(F^*B)\le\mu(\Omega^1_X)$ and
$$\mu(Z_X^1)=\frac{p^2-p+2}{p^2+1}\mu(\Omega^1_X)$$ we have
$$\mu(B)-\mu(Z_X^1)\le
-\frac{(p^2-p+1)(p-1)}{p(p^2+1)}\mu(\Omega^1_X).$$

If $V_1\cap F^*B\neq 0$, then ${\rm rk}(B)=2$, $V_2\cap F^*B=0$ and
$$\sF_1:=\frac{V_1\cap F^*B}{V_2\cap F^*B}\subset
\frac{V_1}{V_2}\,,\quad \sF_0:=\frac{V_0\cap F^*B}{V_1\cap
F^*B}\subset \frac{V_0}{V_1}$$ are subsheaves of rank $1$. On the
other hand, by a theorem of Ilangovan-Mehta-Parameswaran (cf.
Section 6 of \cite{L} for the precise statement): If $E_1$, $E_2$
are semi-stable bundles with ${\rm rk}(E_1)+{\rm rk}(E_2)\le p+1$,
then $E_1\otimes E_2$ is semi-stable. We see that
$$V_1/V_2=\Omega_X^1\otimes\Omega_X^1\,,\quad V_0/V_1=\Omega^1_X$$
are semi-stable since $p>2$. Thus \ga{3.14}
{\mu(B)=\frac{\mu(\sF_0)+\mu(\sF_1)}{2p}\le
\frac{\mu(\Omega_X^1)+\mu(\Omega^1_X\otimes\Omega^1_X)}{2p}} which
implies that
$$\aligned\mu(B)-\mu(Z_X^1)&=\mu(B)-\frac{p^2-p+2}{p^2+1}\mu(\Omega^1_X)\\
&\le\frac{3\,\mu(\Omega_X^1)}{2p}-\frac{p^2-p+2}{p^2+1}\mu(\Omega^1_X)\\
&=-\frac{p^2(2p-5)+4p-3}{2p(p^2+1)}\mu(\Omega^1_X).
\endaligned$$

\end{proof}

\bibliographystyle{plain}

\renewcommand\refname{References}

\end{document}